\documentclass[11pt,a4paper]{article}

\usepackage[utf8]{inputenc}
\usepackage[T1]{fontenc}
\usepackage{lmodern}
\usepackage{geometry}
\usepackage{amsmath,amssymb,amsthm,mathtools}
\usepackage{enumitem}
\usepackage{mathrsfs}
\usepackage{hyperref}

\usepackage[style=ieee, isbn=false, url=false]{biblatex}
\newtheorem{theorem}{Theorem}[section]

\newtheorem{proposition}[theorem]{Proposition}

\newtheorem{remark}[theorem]{Remark}

\DeclareMathOperator{\supp}{supp}

\newcommand{\R}{\mathbb{R}}

\title{Fractional Laplacian in bended strip}
\author{Fedor Bakharev\thanks{St.Petersburg State University, Universitetskaya emb. 7-9, St.Petersburg, 199034, Russia, e-mail: f.bakharev@spbu.ru}
\ and
Sergey Matveenko\thanks{Chebyshev Laboratory, St. Petersburg State University, 14th Line V.O., 29, Saint Petersburg 199178 Russia, e-mail: matveis239@gmail.com}
}

\begin{document}

\maketitle

\noindent{\bf Abstract.} 
The spectral properties of the restricted fractional Laplacian with Dirichlet boundary conditions in a smoothly bent waveguide is investigated. The existence of eigenvalues below the threshold of the continuous spectrum is proved, generalizing classical results known for the local Laplace operator. Our approach utilizes the Caffarelli--Silvestre extension, addressing the specific geometric difficulties arising from the operator non-locality. The sufficient conditions on the curvature magnitude and distribution to ensure the existence of these trapped modes is established.

\medskip


\medskip

\noindent{\bf AMS classification codes:} Primary: 35R11, Secondary: 81Q10.

\section{Introduction}

The primary objective of this work is to generalize the classical spectral theory of Dirichlet Laplacians in bent waveguides to the setting of non-local operators, specifically the fractional Laplacian.

For the standard Dirichlet Laplacian in quantum waveguides (domains with cylindrical outlets to infinity) a rich theory exists describing how geometric deformations affect the spectrum. 
It is well-known that a straight quantum waveguide possesses a purely continuous spectrum bounded from below by the first eigenvalue of the cross-section. However, geometric perturbations such as bending, twisting, or local enlargement can 
induce 
discrete eigenvalues below this threshold, representing "trapped modes" or bound states. The foundational works in this area (see, e.g., \cite{DuEx1995, ExSe89, GoJa1992, ESS}) established that for purely bent waveguides, eigenvalues emerge regardless of the smallness of the curvature. Conversely, twisting the waveguide often creates a repulsive effect that may prevent the existence of bound states \cite{EK, GoJa1992}.

In contrast, the spectral theory for non-local operators in unbounded domains is significantly less developed, despite their growing importance in physics and probability. The fractional Laplacian $(-\Delta)^\alpha$ naturally arises in the study of relativistic quantum mechanics, where it describes the kinetic energy of particles with negligible mass \cite{CaMaSi90, Nardini86}. From a stochastic perspective, these operators generate L\'evy processes, which generalize Brownian motion by allowing for discontinuous paths (jumps), capturing "heavy-tailed" phenomena in anomalous diffusion \cite{GaSt19, Valdinoci2009}.

Recent progress has been made regarding the fractional Dirichlet problem in unbounded domains. In \cite{BaNa2023JST}, Bakharev and Nazarov characterized the essential spectrum of the fractional Laplacian in domains with cylindrical outlets, proving stability results analogous to the local case. Furthermore, in the case of "broken" or V-shaped waveguides, it was recently shown in \cite{BaMa24} that an eigenvalue emerges below the essential spectrum threshold even for angles arbitrarily close to $\pi$. This result relies heavily on the singularity of the boundary at the corner.

A major ambiguity in the non-local setting is the definition of the Dirichlet condition. One approach defines the operator as the spectral power $(-\Delta_\Omega)^\alpha$ of the conventional Dirichlet Laplacian. This, however, reduces the spectral analysis back to the standard operator. 
Another approach adopted in this paper is the \textit{restricted} fractional Laplacian $\mathcal{A}^\Omega_\alpha$. The operator is defined via the quadratic form:
\[
a_\alpha^\Omega[u] = (\mathcal{A}^\Omega_\alpha u, u) := \int_{\mathbb{R}^n} |\xi|^{2\alpha} |\mathcal{F}_n u(\xi)|^2 \, d\xi,
\]
where $\alpha\in(0, 1)$ and $\mathcal{F}_n$ denotes the $n$-dimensional Fourier transform
\[
\mathcal{F}_n u(\xi) = \frac{1}{(2\pi)^{n/2}} \int_{\mathbb{R}^n} e^{-i \xi \cdot x} u(x) \, dx.
\]
The domain of the form is the space of functions in the global Sobolev space $H^\alpha(\mathbb{R}^n)$ supported within $\overline{\Omega}$:
\[
\operatorname{Dom}(a_\alpha^\Omega) = \widetilde{H}^\alpha(\Omega) := \left\{ u \in H^\alpha(\mathbb{R}^n) : \operatorname{supp} u \subset \overline{\Omega} \right\}.
\]

In this paper, we investigate the existence of the discrete spectrum for $\mathcal{A}^\Omega_\alpha$ in a smoothly bent waveguide. 
The standard technique for analyzing the fractional Laplacian is the Caffarelli--Silvestre extension \cite{CaSi}, which maps the non-local problem in $\mathbb{R}^n$ to a local problem in the half-space $\mathbb{R}^{n+1}_+$.
For the V-shaped waveguide, one can construct this extension on parallel cross-sections. However, for a curved waveguide, one must utilize cross-sections orthogonal to the waveguide axis. The crucial difficulty is that the extensions of these cross-sections into the extra dimension intersect if the curvature radius is small relative to the extension depth. This intersection implies that the "local" extension variables become entangled, reflecting the non-local interaction between distant parts of the waveguide.

Our main result demonstrates that discrete eigenvalues exist for the restricted fractional Laplacian in a bent waveguide 
under specific geometric constraints.
Unlike classical results which often rely solely on the smallness of the curvature $\|\kappa\|_\infty$, our analysis requires two parameters: one controlling the magnitude of the curvature and another controlling its support (distribution) along the waveguide. By carefully estimating the interaction strength between cross-sections in the extended space, we construct a trial function for the Rayleigh quotient that falls below the continuous spectrum threshold.

The paper is organized as follows. Section~\ref{sec:Caff-Sil} reviews the necessary tools regarding the Caffarelli--Silvestre extension. Section~\ref{sec:geometry} provides a detailed description of the waveguide geometry and the coordinate systems used. Section~\ref{sec:CSenergy} derives the auxiliary energy estimates required to handle the geometric overlap of the extensions. Finally, Section~\ref{sec:existence} contains the proof of the main result: the existence of an eigenvalue below the threshold of the essential spectrum.

\section{Caffarelli -- Silvestre extension}
\label{sec:Caff-Sil}

The connection between fractional differential operators and generalized harmonic extensions was established in \cite{MO} and became  popular due to the work \cite{CaSi}. Specifically, for a function $u$ belonging to $\widetilde{H}^\alpha(\Omega)$, the function
\begin{equation*}
    \label{def-Caf-Sil}
U(x, y) = \int_{\mathbb R^n}{\mathcal P}_\alpha(x - \widetilde{x}, y)u(\widetilde{x})\,d\widetilde{x}
\end{equation*}
with the generalized Poisson kernel
\begin{equation*}
    \label{def-Poisson-kernel}
{\mathcal P}_\alpha(x, y) = \frac{\Gamma\left(\frac{n+2\alpha}2\right)}{\pi^{\frac  n2}\Gamma(\alpha)}\frac{y^{2s}}{(|x|^2 + y^2)^{\frac n2 + \alpha}}
\end{equation*}
is called the Caffarelli -- Silvestre extension of $u$. The function $U$ is a minimizer of the weighted Dirichlet integral 
\begin{equation*}
    \label{def-Dir-int}
{\mathcal E}_\alpha^\Omega(W) = \int_0^\infty\int_{\mathbb R^n}y^{1-2\alpha}|\nabla W(x, y)|^2dx dy
\end{equation*}
over the set 
\begin{equation*}
        {\mathcal W}_\alpha^\Omega(u) = \{W = W(x, y)\colon {\mathcal E}_\alpha^\Omega(W) < \infty,\ W|_{y=0}=u\}.
\end{equation*}
Moreover, the Caffarelli -- Silvestre extension is a solution of the boundary problem
\begin{equation*}
\label{def-Caf-Sil-prolem}
    - \mathop{\rm div}\nolimits y^{1-2\alpha}\nabla U = 0, \qquad U(x, 0) = u(x).
\end{equation*}
The following identity provide the Dirichlet fractional Laplacian of $u$ via its Caffarelli -- Silvestre extension
\begin{equation*}
    {\mathcal A}_\alpha^\Omega u = - c_\alpha\lim\limits_{y\to0+} y^{1-2\alpha}\partial_y U(\cdot, y), \quad \text{where} \quad c_\alpha = \frac{4^\alpha\Gamma(\alpha+1)}{2\alpha\Gamma(1-\alpha)}.
\end{equation*}
At  the same time the weighted Dirichlet integral is proportional to the quadratic form
\begin{equation*}
    \label{def-form-via-energy-repr}
    a_\alpha^\Omega[u]  = c_\alpha{\mathcal E}_\alpha^\Omega(U)
\end{equation*}
and hence the spectral problem for the eigen pair $(\lambda, u)$ of the Dirichlet restricted Laplacian admits the following variational statement:
\begin{equation*}
     c_\alpha\int_0^\infty\int_{\mathbb R^n}y^{1-2\alpha}\nabla U(x, y) \nabla V(x, y) dx dy = \lambda \int_{\R^n} u(x) v(x) dx, \quad \forall v \in \widetilde{H}^\alpha(\Omega),
\end{equation*}
where $U$ and $V$ are the Caffarelli -- Silvestre extensions of $u$ and $v$, respectively.

\section{Geometry of the bent strip and tubular coordinates}
\label{sec:geometry}

We work in \(\R^2\) with tubular coordinates \(x=(s,n)\) built around a reference curve.  
Let \(\gamma\in C^2(\R;\R^2)\) be an arc–length parametrized curve, \(|\dot\gamma|=1\). Set
\[
\tau(s)=\dot\gamma(s),\qquad 
\nu(s)=(-\dot\gamma_2(s),\dot\gamma_1(s)),\qquad 
\kappa(s)=\dot\gamma_2(s)\ddot\gamma_1(s)-\dot\gamma_1(s)\ddot\gamma_2(s),
\]
so that the Frenet formulas hold
\[
\dot\tau(s)=-\kappa(s)\,\nu(s),\qquad \dot\nu(s)=\kappa(s)\,\tau(s),
\]
and \((\tau,\nu)\) is an orthonormal positively oriented frame.

Assume
\begin{equation*}\label{eq:assumptions}
\kappa\in C^1_c(\R),\qquad \supp\kappa\subset[-\ell,\ell],
\end{equation*}
and choose \(\rho>a>0\) so that the normal coordinate map
\[
\mathscr X:\Pi_\rho:=\R\times(-\rho,\rho)\to\R^2,\qquad 
\mathscr X(s,n)=\gamma(s)+n\,\nu(s),
\]
is a \(C^2\) diffeomorphism onto its image \(\Omega_\rho:=\mathscr X(\Pi_\rho)\).
A convenient sufficient condition excluding focal points is
\[
0<\rho<\|\kappa\|_\infty^{-1}\quad\Longrightarrow\quad J(s,n):=1+\kappa(s)n\ge 1-\rho\|\kappa\|_\infty>0 \ \text{on }\Pi_\rho.
\]
(Injectivity of normal rays also holds for all sufficiently small \(\rho\); we henceforth fix such a \(\rho>a\).)
The physical bent strip of half–width \(a\) is
\[
\Omega_a:=\{\mathscr X(s,n):\ |n|<a\}\subset\Omega_\rho,
\]
and we refer to \(\Omega_\rho\) as the embedded superstrip (the wider tubular neighborhood used for test functions and cutoffs).

Differentiating \(\mathscr X\) and using the Frenet identities yields
\[
\partial_s\mathscr X=(1+\kappa n)\,\tau,\qquad \partial_n\mathscr X=\nu.
\]
Hence the coordinate lines are orthogonal and the metric \(g=D\mathscr X^\top D\mathscr X\) reads
\begin{equation*}\label{eq:metric}
g_{ss}(s,n)=(1+\kappa(s)n)^2,\qquad g_{sn}(s,n)=0,\qquad g_{nn}(s,n)=1,
\end{equation*}
with inverse
\[
g^{ss}(s,n)=(1+\kappa(s)n)^{-2},\qquad g^{sn}=0,\qquad g^{nn}=1,
\]
and Jacobian
\[
\det D\mathscr X(s,n)=1+\kappa(s)n=J(s, n).
\]
Consequently, for smooth \(u=u\circ\mathscr X\),
\[
|\nabla_x u|^2 = g^{ss}|\partial_s u|^2 + g^{nn}|\partial_n u|^2
= J^{-2} |\partial_s u|^2+|\partial_n u|^2,
\]
and the area element is \(\,dx=J\,dsdn\).

On \(\Omega_\rho\times\R_+\) we consider the weighted measure \(y^{1-2\alpha}\,dxdy\).
The pullback by \(\mathscr X\) followed by multiplication by \(J^{1/2}\) defines the \emph{unitary flattening}
\begin{equation*}\label{eq:unitary}
\widetilde W(s,n,y)=J(s,n)^{1/2}\,W(\mathscr X(s,n),y),
\end{equation*}
which is an isometry
\[
\int_{\Omega_\rho} \int_0^\infty y^{1-2\alpha}|W|^2\,dx dy
=\int_{\Pi_\rho}\int_0^\infty y^{1-2\alpha}|\widetilde W|^2\,ds dn dy.
\]
We will use test extensions supported in \(\Omega_\rho\times\R_+\), equivalently \(\widetilde W\) with compact support in \(s\) and \(y\), and with homogeneous Dirichlet condition \(\widetilde W(\cdot,\pm\rho,\cdot)=0\) at the artificial boundary of the superstrip.

\section{CS energy in the embedded superstrip and unitary flattening}
\label{sec:CSenergy}

\begin{proposition}
Let \(\widetilde W\in C_c^\infty(\Pi_\rho \times[0,\infty))\) with \(\widetilde W(\cdot,\pm\rho,\cdot)=0\).
Then
\begin{equation}\label{eq:square}
\begin{aligned}
c_\alpha\int_0^\infty\int_{\Omega_\rho} y^{1-2\alpha}|\nabla W|^2\,dx\,dy
= c_\alpha \int_0^\infty\int_{\Pi_\rho} y^{1-2\alpha}\Bigg(
&J^{-2}\big|\partial_s\widetilde W+A\,\widetilde W\big|^2
+\big|\partial_n\widetilde W\big|^2 \\
&\qquad\qquad\qquad +\,|\partial_y\widetilde W|^2 + V |\widetilde W|^2\Bigg)\,dn\,ds\,dy,
\end{aligned}
\end{equation}
with
\[
A(s,n)=-\frac{n \kappa'(s)}{2 J(s,n)},\qquad 
V(s,n)=-\frac{|\kappa(s)|^2}{4 |J(s,n)|^2}.
\]
\end{proposition}

\begin{remark}
The curvature produces the attractive potential \(V\le 0\). Outside the bending region \(|s|>\ell\) we have \(\kappa\equiv0\), hence \(A=0\) and \(V=0\), and the integrand coincides with the straight one.
\end{remark}

\begin{proof}
Let \(\overline{W}(s, n, y):=W(\mathscr X(s,n),y)=J^{-1/2}\widetilde W(s,n,y)\) and write
\[
\int_{\Omega_\rho} |\nabla_x W|^2\,dx
=\int_{\Pi_\rho} \Big[g^{ss}|\partial_s \overline{W}|^2+g^{nn}|\partial_n \overline{W}|^2\Big] J\,dsdn
=\int_{\Pi_\rho}\Big[J^{-2}|\partial_s \overline{W}|^2+|\partial_n \overline{W}|^2\Big] J\,dsdn.
\]
Moreover,
\[
\int_{\Omega_\rho} |\partial_y W|^2\,dx
=\int_{\Pi_\rho} |\partial_y \overline{W}|^2 J\,dsdn=\int_{\Pi_\rho} |\partial_y \widetilde{W}|^2 \,dsdn.
\]

For the \(s\)- and \(n\)-derivatives we compute
\[
\partial_s \overline{W}=\partial_s(J^{-1/2}\widetilde W)
=J^{-1/2}\Big(\partial_s\widetilde W-\frac{n \kappa'}{2J}\widetilde W\Big)
=J^{-1/2}\big(\partial_s\widetilde W+A \widetilde W\big),
\]
\[
\partial_n \overline{W}=\partial_n(J^{-1/2}\widetilde W)
=J^{-1/2}\Big(\partial_n\widetilde W-\frac{\kappa}{2J}\widetilde W\Big).
\]
Expand the \(n\)-term and integrate by parts in \(n\) (with \(s,y\) as parameters). Using
\[
\frac{\kappa}{J} \widetilde{W} \partial_n \widetilde{W} =
\frac{1}{2}\partial_n\Big(\frac{\kappa}{J}\widetilde{W}^2\Big)
+\frac{\kappa^2}{2J^2} \widetilde{W}^2 \,
\]
the boundary condition \(\widetilde W(\cdot,\pm\rho,\cdot)=0\), we obtain
\[
\int_{-\rho}^{\rho}\Big|\partial_n\widetilde W-\frac{\kappa}{2J}\widetilde W\Big|^2dn
=\int_{-\rho}^{\rho}\Big(|\partial_n\widetilde W|^2 
-\frac{\kappa}{J} \partial_n\widetilde W \cdot \widetilde{W}
+\frac{\kappa^2}{4J^2}|\widetilde W|^2\Big) \,dn
\]
\[
=\int_{-\rho}^{\rho}\Big(|\partial_n\widetilde W|^2 
-\frac{1}{2}\partial_n\Big(\frac{\kappa}{J}\widetilde{W}^2\Big)
-\frac{\kappa^2}{4J^2}|\widetilde W|^2\Big)dn
=\int_{-\rho}^{\rho}\Big(|\partial_n\widetilde W|^2 
-\frac{\kappa^2}{4J^2}|\widetilde W|^2\Big)dn.
\]
Collecting the \(s\)-, \(n\)-, and \(y\)-contributions and inserting the common weight \(y^{1-2\alpha}\) proves \eqref{eq:square}.
\end{proof}

\section{Existence of an eigenvalue below $\lambda_1$}
\label{sec:existence}
We construct a quasimode with Rayleigh quotient strictly below $\lambda_1$. Fix an even function $\chi\in C_c^\infty(\R)$ with
\begin{equation}\label{eq:zeta-profile}
  0\le \chi\le 1,\qquad
  \chi\equiv 1  \text{ on } [-1/2,1/2],\qquad
  \supp\chi\subset (-1,1).
\end{equation}
For parameters $L\ge 1$ and $\rho>a>0$ define the rescaled cutoffs
\begin{equation}\label{eq:psi-chi-def}
  \chi_L(s):=\chi (s/L),\qquad
  \chi_\rho(n):=\chi (n/\rho).
\end{equation}
We will assume $\rho\ge 2a$, so that $\chi_\rho\equiv 1$ on $[-a,a]$ and $\chi_\rho(\pm\rho)=0$.
The scaling of derivatives gives the following relations
\begin{equation}\label{eq:psi-chi-scales}
  \|\chi_L\|_{L^2(\R)}^2=L \|\chi\|_{L^2(\R)}^2,\quad
  \|\chi_L'\|_{L^2(\R)}^2=L^{-1} \|\chi'\|_{L^2(\R)}^2,\quad
  \|\chi_\rho'\|_{\infty} = \rho^{-1}\|\chi'\|_{\infty}.
\end{equation}

\medskip

Let $u_1$ be the first Dirichlet eigenfunction on $(-a,a)$. We fix the normalization
\[\int_{-a}^a |u_1(n)|^2\,dn=1.\]
Let $U_1$ denote its Caffarelli–Silvestre extension in $(n,y)\in\R\times(0,\infty)$. We set
\[
  \widetilde W(s,n,y):=\chi_L(s) \chi_\rho(n) U_1(n,y),
  \qquad
  \widetilde w(s,n):=\chi_L(s) \chi_\rho(n) u_1(n)=\chi_L(s) u_1(n).
\]
Since $\chi_\rho\equiv1$ on $[-a,a]$ and $u_1$ is supported in $[-a,a]$, the $L^2$ “mass”
factorizes as
\begin{equation}\label{eq:L2-mass}
  \|\widetilde w\|_{L^2(\Pi_\rho)}^2
  = \Big(\int_{-a}^a |\chi_\rho(n)|^2 |u_1(n)|^2\,dn\Big) \|\chi_L\|_{L^2(\R)}^2
  = \|\chi_L\|_{L^2(\R)}^2 
  = L^{2} \|\chi\|_{L^2(\R)}^2.
\end{equation}
For later use, introduce two finite values
\begin{align}\label{eq:K-mass}
  &K(\rho, \chi):=\int_0^\infty \int_{-\rho}^{\rho}
           y^{1-2\alpha} |\chi_\rho(n)|^2 |U_1(n,y)|^2\,dn dy \in (0,\infty), \\
           \label{eq:K'-mass}
  &K'_\rho(\rho, \chi):=\int_0^\infty \int_{-\rho}^{\rho}
           y^{1-2\alpha} |\chi_\rho'(n)|^2 |U_1(n,y)|^2\,dn dy \in (0,\infty).
\end{align}

\paragraph{(i) Upper bound for the $s$–term.}
On the set $\{|n|\le \rho\}$ we have
\begin{equation}\label{eq:J-A-rho}
  J^{-2}(s,n)\le \bigl(1-\rho\|\kappa\|_\infty\bigr)^{-2}=: J_+(\rho),\quad
  |A(s,n)|\le \frac{\rho}{2\,(1-\rho\|\kappa\|_\infty)} |\kappa'(s)|=:C_A(\rho)|\kappa'(s)|.
\end{equation}
Thus, using Young’s inequality, \eqref{eq:J-A-rho}, and \eqref{eq:psi-chi-scales},
\begin{align}
\mathcal{I}_s
&:=c_\alpha \int_0^\infty\int_{\Pi_\rho}
   y^{1-2\alpha}J^{-2}\bigl|\partial_s\widetilde{W}+A \widetilde{W}\bigr|^2
   \,dndsdy \nonumber\\
&\le 2c_\alpha J_+(\rho) \int_0^\infty \int_{\Pi_\rho}
   y^{1-2\alpha}|\chi_\rho|^2 |U_1|^2
   \Big(|\chi_L'|^2 + |A|^2 |\chi_L|^2\Big)\,dn ds dy \nonumber\\
&\le 2c_\alpha J_+(\rho) K(\rho, \chi)
   \Big(\|\chi_L'\|_{L^2(\R)}^2
       + |C_A(\rho)|^2 \int_{\R} |\kappa'(s)|^2\,ds\Big),
\label{eq:S-bound-unified}
\end{align}
where $K(\rho, \chi)$ is defined in \eqref{eq:K-mass}. We used that $\supp\kappa'\subset[-\ell,\ell]$
and $\chi_L\equiv1$ on $[-\ell,\ell]$ for $L\ge 2\ell$.

Dividing by the mass identity \eqref{eq:L2-mass}, and using \eqref{eq:psi-chi-scales} gives, for $L\ge 2\ell$,
\begin{equation}\label{eq:S-final-unified}
  \frac{\mathcal{I}_s}{\|\widetilde{w}\|_{L^2(\Pi_\rho)}^2}
   \leq \frac{A_0(\rho,\chi)}{L^2}
        + \frac{B_0(\rho,\chi)}{L},
\end{equation}
with
\[
  A_0(\rho, \chi):=\frac{2c_\alpha J_+(\rho)K(\rho, \chi)}{\|\chi\|_{L^2}^2} \|\chi'\|_{L^2}^2,
  \qquad
  B_0(\rho, \chi):=\frac{2c_\alpha J_+(\rho)C_A(\rho)^2K(\rho, \chi)}{\|\chi\|_{L^2}^2} \|\kappa'\|_{L^2(-\ell,\ell)}^2.
\]

\paragraph{(ii) Localized CS identity in $(n,y)$.}
Let $\varphi(s,n):=\chi_L(s)\chi_\rho(n)$. Since $U_1=U_1(n,y)$ is independent of $s$, we use the weak formulation of the extension with the test function $U_1 \varphi^2$:
\begin{equation}\label{eq:cs-test}
c_\alpha \int_0^\infty \int_{\R}
  y^{1-2\alpha}\nabla_{n,y}U_1\cdot\nabla_{n,y} \big(U_1\varphi^2\big)\,dndy
  = \lambda_1 \int_{\R} |u_1(n)|^2 |\varphi(n,s)|^2\,dn.
\end{equation}
Because $\nabla_{n,y}\varphi=(\chi_L\,\chi_\rho',\,0)$, expanding the left-hand side of
\eqref{eq:cs-test} gives, for each $s$,
\[
  c_\alpha \int y^{1-2\alpha}\varphi^2\big(|\partial_n U_1|^2+|\partial_y U_1|^2\big)\,dndy
  + 2c_\alpha \int y^{1-2\alpha}\chi_L^2 \chi_\rho \chi_\rho' U_1 \partial_n U_1\,dn dy
  = \lambda_1 \int \chi_L^2\chi_\rho^2 |u_1|^2\,dn.
\]
Integrating in $s$ and recalling $\partial_n\widetilde{W}=\chi_L(\chi_\rho\partial_n U_1+\chi_\rho' U_1)$,
$\partial_y\widetilde{W}=\chi_L\chi_\rho\partial_y U_1$, we obtain the identity
\begin{equation}\label{eq:IMS-weighted}
\begin{aligned}
 \mathcal{I}_{ny}:= &c_\alpha \int_0^\infty \int_{\Pi_\rho}
   y^{1-2\alpha}\Big(|\partial_n\widetilde{W}|^2+|\partial_y\widetilde{W}|^2\Big)\,dn ds dy
   - \lambda_1 \int_{\Pi_\rho}|\widetilde{w}|^2\,dn ds \\
  &= c_\alpha  \int_{\R}|\chi_L(s)|^2\,ds
  \cdot \int_0^\infty \int_{-\rho}^{\rho}
   y^{1-2\alpha} |\chi_\rho'(n)|^2 |U_1(n,y)|^2\,dn dy.
\end{aligned}
\end{equation}
Using \eqref{eq:K'-mass} we can rewrite \eqref{eq:IMS-weighted} as
\[
  c_\alpha \int y^{1-2\alpha}\big(|\partial_n\widetilde{W}|^2+|\partial_y\widetilde{W}|^2\big)
  - \lambda_1 \int |\widetilde{w}|^2
  = c_\alpha \|\chi_L\|_{L^2(\R)}^2 K'(\rho, \chi).
\]
Dividing by mass identity and keeping in mind that $\supp\,\chi_\rho'\subset\{\rho/2<|n|<\rho\}$, we have the bound
\begin{equation}\label{eq:E-rho-bound}
 \frac{\mathcal{I}_{n,y}}{\|\widetilde{w}\|_{L^2(\Pi_\rho)}^2}=c_\alpha K'(\rho, \chi)
  \leq c_\alpha \frac{\|\chi'\|_{\infty}^2}{\rho^2}
  \int_0^\infty \int_{\rho/2<|n|<\rho} y^{1-2\alpha} |U_1(n,y)|^2\,dn dy.
\end{equation}

\paragraph{(iii) Curvature gain.}
From $J(s,n)\ge 1-\rho\|\kappa\|_\infty$ on $|n|\le\rho$,
\begin{align}
\mathcal{I}_V
&:= c_\alpha \int_0^\infty \int_{\Pi_\rho}
   y^{1-2\alpha} V(s,n) |\widetilde{W} |^2\,dn ds dy \nonumber\\
&= -\frac{c_\alpha}{4} \int_{\R} |\kappa(s)|^2 |\chi_L(s)|^2
      \left(\int_{-\rho}^{\rho}\frac{|\chi_\rho(n)|^2}{|J(s,n)|^2}
            \int_0^\infty y^{1-2\alpha}|U_1(n,y)|^2\,dy dn\right) ds \nonumber\\
&\le -\frac{c_\alpha}{4(1+\rho\|\kappa\|_\infty)^2}
    \Big( \int_0^\infty \int_{-\rho}^{\rho} y^{1-2\alpha}|\chi_\rho(n)|^2 |U_1(n,y)|^2\,dn dy\Big)
    \int_{\R}|\kappa(s)|^2 |\chi_L(s)|^2\,ds \nonumber\\
&= -\frac{c_\alpha K(\rho, \chi)}{4(1+\rho\|\kappa\|_\infty)^2}
    \int_{-\ell}^{\ell}|\kappa(s)|^2\,ds,
\label{eq:V-final-unified}
\end{align}
where $K(\rho, \chi)$ is defined in \eqref{eq:K-mass}.
Dividing by the mass identity \eqref{eq:L2-mass}, we obtain
\begin{equation}\label{eq:V-over-mass-unified}
  \frac{\mathcal{I}_V}{\|\widetilde{w}\|_{L^2(\Pi_\rho)}^2}
   \le - \frac{C_0'(\rho,\chi)}{L},\quad
  C_0'(\rho,\chi):=\frac{c_\alpha K(\rho, \chi)}{4 \|\chi\|_{L^2}^2\,(1+\rho\|\kappa\|_\infty)^2} \|\kappa\|_{L^2(-\ell,\ell)}^2.
\end{equation}

\paragraph{(iv) Rayleigh quotient and choice of parameters.}
Combining the estimates we get, for $L\ge 2\ell$,
\begin{equation}\label{eq:RQ-master-unified}
  \mathcal Q[\widetilde{W}]-\lambda_1 \le  \frac{A_0(\rho,\chi)}{L^2}
       + \frac{B_0(\rho,\chi)}{L}
       + c_\alpha K_\rho' (\rho, \chi)
       - \frac{C_0'(\rho,\chi)}{L}.
\end{equation}

Because $C_0'(\rho,\chi)$ and $B_0(\rho,\chi)$ share the same factor
$K(\rho, \chi)/\|\chi\|_{L^2}^2$, the inequality $C_0'(\rho,\chi)>B_0(\rho,\chi)$ is
equivalent to a $\kappa$–only condition:
\begin{equation}\label{eq:SB-criterion}
 \|\kappa\|_{L^2(-\ell,\ell)}^2 >
 \frac{8 J_+(\rho) C_A(\rho)^2}{(1+\rho\|\kappa\|_\infty)^2}\,
 \|\kappa'\|_{L^2(-\ell,\ell)}^2.
\end{equation}
This becomes
\begin{equation}\label{eq:SB-explicit}
 \frac{\|\kappa\|_{L^2(-\ell,\ell)}}{\|\kappa'\|_{L^2(-\ell,\ell)}}
 > \sqrt{2} \rho \frac{1+\rho\|\kappa\|_\infty}{(1-\rho\|\kappa\|_\infty)^2}
\end{equation}
under the geometric constraint $\rho\|\kappa\|_\infty<1$.

\paragraph{(v) Estimates in case $\alpha\in(0, 1/2)$.}
The only change is that the integrals $K(\rho,\chi)$ and $K_\rho'(\rho, \chi)$
can be infinite. To prevent this we set
\[
  \widetilde W(s,n,y):=\chi_\tau(y) \chi_L(s) \chi_\rho(n) U_1(n,y),
  \qquad
  \widetilde w(s,n):=\chi_\tau(y) \chi_L(s) \chi_\rho(n) u_1(n)=\chi_L(s) u_1(n),
\]
where 
\begin{equation}\label{eq:tau-chi-def}
  \chi_\tau(s):=\chi (y/\tau).
\end{equation}
Hence $\nabla_{n,y}\varphi=(\chi_L\,\chi_\tau\chi_\rho',\,\chi_L\,\chi_\tau'\chi_\rho)$ and thus

\begin{equation*}
 \mathcal{I}_{ny}
  = c_\alpha  \int_{\R}|\chi_L(s)|^2\,ds
  \cdot \int_0^\infty \int_{-\rho}^{\rho}
   y^{1-2\alpha} (|\chi_\rho'(n)\chi_\tau(y)|^2 + |\chi_\rho(n)\chi_\tau'(y)|^2)|U_1(n,y)|^2\,dn dy,
\end{equation*}
\begin{equation*}
 \frac{\mathcal{I}_{n,y}}{\|\widetilde{w}\|_{L^2(\Pi_\rho)}^2}=c_\alpha (K_\rho'(\rho, \chi) + K_\tau'(\rho, \chi)),
\end{equation*}
where 
\begin{align}\label{eq:K_tau-mass}
  &K(\rho, \tau, \chi):=\int_0^\infty \int_{-\rho}^{\rho}
           y^{1-2\alpha} |\chi_\tau(y)\chi_\rho(n)|^2 |U_1(n,y)|^2\,dn dy \in (0,\infty), \\
           \label{eq:K_rho'-mass}
  &K'_\rho(\rho,\tau, \chi):=\int_0^\infty \int_{-\rho}^{\rho}
           y^{1-2\alpha} |\chi_\tau(y)\chi_\rho'(n)|^2 |U_1(n,y)|^2\,dn dy \in (0,\infty), \\
           \label{eq:K_tau'-mass}
    &K'_\tau(\rho, \tau, \chi):=\int_0^\infty \int_{-\rho}^{\rho}
           y^{1-2\alpha} |\chi_\tau'(y)\chi_\rho(n)|^2 |U_1(n,y)|^2\,dn dy \in (0,\infty).
\end{align}
 According to the definition we have
$U_1(\cdot, y) = {\mathcal P}_s(\cdot, y) * u_1$. The Young convolution inequality provides the estimate
\begin{equation*}
    \|U_1(\cdot, y); L_2(\mathbb R)\|^2 = \|{\mathcal P}_\alpha(\cdot, y) * u_1; L_2(\mathbb R)\|^2\leqslant \|{\mathcal P}_\alpha(\cdot, u_1); L_2(\mathbb R)\|^2  \|u_1; L_1(\mathbb R)\|^2.
\end{equation*}
Note that $u_1$ is summable because it is square summable and has a compact support. Introducing a new variable $\widetilde{x}=x/y$, we write 
\begin{equation*}
     \|{\mathcal P}_\alpha(\cdot, y); L_2(\mathbb R)\|^2  = \int_\mathbb R \frac{y^{4\alpha}\,dx}{(x^2 + y^2)^{1+2s}}= \frac{1}{y}\int_\mathbb R \frac{d \widetilde{x}}{(1 +  \widetilde{x}^2)^{1+2\alpha}},
\end{equation*}
and thus 
\begin{equation}
\label{eq:U_1-norm-est}
    \|U_1(\cdot, y); L_2(\mathbb R)\|^2 \leqslant C y^{-1}.
\end{equation}
Due to \eqref{eq:psi-chi-def}, \eqref{eq:tau-chi-def} and \eqref{eq:U_1-norm-est} the integrals \eqref{eq:K_tau-mass} - \eqref{eq:K_tau'-mass} are finite and, moreover, the last integral can be estimated as follows
\begin{equation}\label{eq:E-tau-bound}
K_\tau'(\rho, \tau, \chi) \leq C\int_0^{+\infty}|\chi_\tau'(y)|^2y^{-2\alpha}dy
  \leq \widetilde{C} \frac{\|\chi'\|_{\infty}^2}{\tau^{1+2\alpha}}.
\end{equation}
We update the estimate \eqref{eq:RQ-master-unified}:
\begin{equation}\label{eq:RQ-master-unified-2}
  \mathcal Q[\widetilde{W}]-\lambda_1 \le  \frac{A_0(\rho,\tau,\chi)}{L^2}
       + \frac{B_0(\rho,\tau,\chi)}{L}
       + c_\alpha K_\rho' (\rho, \tau, \chi) + \widetilde{c_\alpha} K_\tau'(\rho, \tau, \chi)
       - \frac{C_0'(\rho, \tau, \chi)}{L}.
\end{equation}
Since, choosing $\tau$ large enough, the term $\widetilde{c_\alpha} K_\tau'(\rho, \tau, \chi)$ can be made arbitrary small, the choice of other parameters can be made as in (iv).
\paragraph{(vi) Estimates in case $\alpha=1/2$.}
Since $U_1(x, y) = {\mathcal P}(\cdot, y) * u_1(x)$ and the Fourier transform preserves $L_2$-norm, the following identity is satisfied
\begin{equation*}
\|U_1; L_2(\mathbb R\times(0, \tau))\|^2 = 
\int_0^{\tau}\int_{-\infty}^\infty |{\mathcal F}_1{\mathcal P}_{1/2}(\cdot, y)(\xi)|^2|{\mathcal F}_1 u_1(\xi)|^2\,d\xi\,dy.
\end{equation*}
Direct computations show that
\begin{equation*}
{\mathcal F}_1{\mathcal P}_{1/2}(\cdot, y)(\xi) = \sqrt{\frac{\pi}2}e^{-y|\xi|},
\end{equation*}
thus, we have
\begin{equation*}
\|U_1; L_2(\mathbb R\times(0, \tau))\|^2 =  \sqrt{\frac{\pi}2}
\int_0^{\tau}\int_{-\infty}^\infty e^{-2y|\xi|}|{\mathcal F}_1 u_1(\xi)|^2\,d\xi\,dy.
\end{equation*}
Applying the Fubini theorem, we get
\begin{equation*}
\|U_1; L_2(\mathbb R\times(0, \tau))\|^2 = \sqrt{\frac{\pi}2}
\int_{-\infty}^\infty \frac{1 - e^{-2\tau|\xi|}}{2|\xi|}|{\mathcal F}_1 u_1(\xi)|^2\,d\xi.
\end{equation*}
We decompose the last integral into a sum
\begin{equation*}
J_1 + J_2 + J_3 :=
\left(\int_{|\xi|\leqslant \tau^{-1}} + \int_{ \tau^{-1} \leqslant |\xi| \leqslant 1} + \int_{|\xi| \geqslant 1}\right) \frac{1 - e^{-2\tau|\xi|}}{2|\xi|}|{\mathcal F}_1 u_1(\xi)|^2\,d\xi.
\end{equation*}
The function $u_1$ has a compact support, so its Fourier transform is a smooth function. Introducing a new variable $\widetilde\xi = \tau\xi$, we obtain an estimate for $J_1$ 
\begin{equation*}
J_1 \leqslant C\int_{0}^1\frac{1-e^{-2\widetilde\xi}}{\widetilde\xi}d\widetilde\xi,    
\end{equation*}
with a convergent integral on the right-hand side. The other terms satisfy the inequalities
\begin{equation*}
\begin{aligned}
J_2&\leqslant C\int_{\tau^{-1}}^1\frac{d\xi}{\xi}\leqslant C\log(\tau),\\
J_3&\leqslant C\int_1^\infty |{\mathcal F}_1 u_1(\xi)|^2\,d\xi\leqslant C.
\end{aligned}
\end{equation*}
Adding up the estimates for $J_1$, $J_2$, and $J_3$, we obtain
\begin{equation*}
\label{eq:U_1-norm-est_alpha_half}
\|U_1; L_2(\mathbb R\times(0, \tau))\|^2 \leqslant C\log(\tau).
\end{equation*}

Again Due to \eqref{eq:psi-chi-def}, \eqref{eq:tau-chi-def} and \eqref{eq:U_1-norm-est_alpha_half} the integrals \eqref{eq:K_tau-mass} - \eqref{eq:K_tau'-mass} are finite and  the last integral can be estimated as follows
\begin{equation}\label{eq:E-tau-bound2}
K_\tau'(\rho, \tau, \chi) 
  \leq  \frac{\|\chi'\|_{\infty}^2}{\tau^{2}}\|U_1; L_2(\mathbb R\times(0, \tau))\|^2\leq C\frac{\|\chi'\|_{\infty}^2\log(\tau)}{\tau^{2}}.
\end{equation}
Finally, we have an estimate similar to 
\eqref{eq:RQ-master-unified-2}.



\begin{theorem}
    Given any profile function $\kappa$ with $\supp \kappa\subset [-1, 1]$ and $\|\kappa\|_{\infty} = 1$, then there exist $\mathcal L$ and $\mathcal E$ such that for $\ell > \mathcal L$ and $\varepsilon < \mathcal E$  the discrete spectrum of $\mathcal A_\alpha^{\Omega^{\varepsilon,\ell}}$ in a bended waveguide $\Omega^{\varepsilon,\ell}$ of unit width with curvature equal to $\kappa_{\ell,\varepsilon} = \varepsilon\kappa(s/\ell)$  is nonempty.
\end{theorem}

\begin{proof}
    Fix a single bump $\chi\in C_c^\infty(\R)$ as in \eqref{eq:zeta-profile} and use
$\chi_L(s)=\chi(s/L)$, $\chi_\rho(n)=\chi(n/\rho)$ with the scalings in
\eqref{eq:psi-chi-scales}. 
Set $C_\kappa:=\int_{-1}^1 \kappa(r)^2\,dr$, $C_{\kappa'}:=\int_{-1}^1 |\kappa'(r)|^2\,dr$, 
$\|\kappa_{\ell,\varepsilon}\|_{L^2(-\ell,\ell)}^2=\,\varepsilon^2\ell\,C_\kappa$ and
$\|\kappa_{\ell,\varepsilon}'\|_{L^2(-\ell,\ell)}^2=\,\varepsilon^2\ell^{-1}\,C_{\kappa'}$.
\paragraph{Step 1 (large admissible superstrip and amplitude).}
Pick $\theta\in(0,1)$, choose any $\rho\ge 2$, and set
\[
  \varepsilon:=\frac{\theta}{\rho}\qquad\Longrightarrow\qquad
  \rho\,\|\kappa_{\ell,\varepsilon}\|_\infty=\theta<1.
\]
Then $J_+(\rho)=(1-\theta)^{-2}$ and $C_A(\rho)=\rho/(2(1-\theta))$.

\paragraph{Step 2 (bending scale to ensure $C_0'>B_0$).}
With the above parametrization, the ratio $C_0'/B_0$ computed from
\eqref{eq:S-final-unified} and \eqref{eq:V-over-mass-unified} simplifies to
\[
  \frac{C_0'}{B_0}
  \;=\;\frac{(1-\theta)^4}{2(1+\theta)^2}\cdot\frac{C_\kappa}{C_{\kappa'}}\cdot\frac{\ell^2}{\rho^2}.
\]
Thus $C_0'>B_0$ holds whenever
\begin{equation}\label{eq:ell-over-rho-threshold}
  \frac{\ell}{\rho}\ > \sqrt{\frac{2(1+\theta)^2}{(1-\theta)^4}\cdot\frac{C_{\kappa'}}{C_\kappa}}\;.
\end{equation}

\paragraph{Step 3 (control the $s$–cutoff error).}
Using \eqref{eq:RQ-master-unified}, choose
\[
  L\ \ge\ L_\ast:=\frac{4A_0(\rho,\chi)}{C_0'(\rho,\chi)-B_0(\rho,\chi)}
  \qquad\text{and}\qquad L\ge 2\ell,
\]
so that $\frac{A_0}{L^2}\le \frac{C_0'-B_0}{4L}$ and $\chi_L\equiv1$ on $[-\ell,\ell]$.
Then
\[
  \mathcal Q[\widetilde W_L]-\lambda_1
  \ \le\ -\,\frac{C_0'(\rho,\chi)-B_0(\rho,\chi)}{4L}\ +\ c_\alpha\,\mathcal E_\rho(\chi).
\]

\paragraph{Step 4 (suppress the IMS remainder).}
By \eqref{eq:E-rho-bound},
\[
  \mathcal E_\rho(\chi)\ \le\ \frac{\|\chi'\|_{L^\infty}^2}{\rho^2}
  \int_0^\infty\!\!\!\int_{\rho/2<|n|<\rho} y^{1-2\alpha}\,U_1(n,y)^2\,dn\,dy
  \ \xrightarrow[\rho\to\infty]{}\ 0.
\]
Since $\rho$ was free in Step~1, enlarge $\rho$ (adjusting $\varepsilon=\theta/\rho$ and
$\ell$ to keep \eqref{eq:ell-over-rho-threshold}) so that
\[
  c_\alpha\,\mathcal E_\rho(\chi)\ \le\ \frac{C_0'(\rho,\chi)-B_0(\rho,\chi)}{8L}.
\]
Therefore
\[
  \mathcal Q[\widetilde W_L]-\lambda_1
  \ \le\ -\,\frac{C_0'(\rho,\chi)-B_0(\rho,\chi)}{8L}\ <\ 0,
\]
and  since due to estimates \eqref{eq:E-tau-bound} and \eqref{eq:E-tau-bound2} a reminder term can also be made an arbitrary small a bound state below $\lambda_1$ exists by the variational principle.

\end{proof}

\paragraph*{Acknowledgments.} 
The results were obtained under support of the Russian Science Foundation (RSF) grant 19-71-30002.

\printbibliography

\end{document}